\documentclass[12pt]{amsart}
\usepackage{fullpage}
\usepackage{hyperref}
\usepackage{amsmath}
\usepackage{amsthm}
\usepackage{amssymb}
\usepackage{tikz-cd}
\usepackage[backend=biber,maxbibnames=99,sorting=anyt,doi=false,isbn=false,url=false]{biblatex}
\renewbibmacro{in:}{}
\bibliography{References} 
\usepackage[inline]{enumitem}
\usepackage{color}
\newtheorem{theorem}{Theorem}[section]
\newtheorem*{theorem*}{Theorem}
\newtheorem*{corollary*}{Corollary}

\newtheorem{proposition}[theorem]{Proposition}
\newtheorem*{proposition*}{Proposition}
\newtheorem{lemma}[theorem]{Lemma}
\theoremstyle{definition}
\newtheorem{definition}[theorem]{Definition} 
\newtheorem*{definition*}{Definition}
\newtheorem*{example*}{Example}
\theoremstyle{remark}

\numberwithin{equation}{section}

\DeclareMathOperator{\Int}{Int}
\DeclareMathOperator{\kmax}{kmax}
\DeclareMathOperator{\kmin}{kmin}
\DeclareMathOperator{\rank}{rank}

\newcommand{\abs}[1]{\lvert#1\vert}
\newcommand{\norm}[1]{\lVert#1\Vert}

\newcommand{\isom}{\cong}

\newcommand{\R}{\mathbb{R}}

\newcommand{\N}{\mathbb{N}}

\newcommand{\hp}{\R^2_\leq}
\newcommand{\ohp}{\R^2_<}
\newcommand{\ehp}{\overline{\R}^2_\leq}
\newcommand{\oehp}{\overline{\R}^2_<}
\newcommand{\diag}{\Delta}
\newcommand{\ediag}{\overline{\Delta}}
\newcommand{\Wpt}{W_p^\triangle}

\title{A Hilbert space embedding of persistence diagrams and barcodes}
\author{Peter Bubenik}
\address{University of Florida, Department of Mathematics}
\email{peter.bubenik@ufl.edu}
\begin{document}
  
\begin{abstract}
  For $p \in [1,\infty]$, we show that the persistence landscape gives 1-Lipschitz embeddings of metric spaces of countable persistence diagrams and barcodes with p-Wasserstein distances into an $L^p$ space. 
\end{abstract}

\subjclass{55N31}

\maketitle

\section{Introduction} \label{sec:introduction}

If each of the vector spaces in a persistence module~\cite{Virk:book} is finite dimensional, then it is isomorphic to a direct sum of interval modules~\cite{crawley}.
The collection of these intervals is called a \emph{barcode}.
Replacing intervals with their infimum and supremum, we obtain a collection of ordered pairs called a \emph{persistence diagram}.
Barcodes and persistence diagrams have a canonical one-parameter family of distances, the $p$-Wasserstein distances, where $1 \leq p \leq \infty$, which depend on a choice of distance for intervals~\cite{BE2022,MR4496687}. 
We introduce a one-parameter family of such distances $d_q$, where $1 \leq q \leq \infty$ (Definition~\ref{def:dq}). 
Let $\Wpt$ denote the $p$-Wasserstein distance that uses the distance $d_p$.
When $p=\infty$, $\Wpt$ equals the bottleneck distance~\cite{bottleneck} (Proposition~\ref{prop:bottleneck}). 
When $p=1$, $\Wpt$ is the recently introduced rank-based Wasserstein distance~\cite{bubenik2026rankbaseddistanceintervalmodules} (Proposition~\ref{prop:rank}).

Consider the barcode $B$ given by a sequence of intervals $(I_j)_{j=1}^n$ or $(I_j)_{j=1}^\infty$, where the interval $I_j$ has the form $[b_j,d_j)$, where $-\infty < b_j < d_j < \infty$. 
For an interval $I = [b,d)$, let the \emph{triangle function} of $I$ be the function $\triangle_I: \R \to [0,\infty)$ given by the distance to the complement of $I$.
See the left side of Figure~\ref{fig:triangle}.
If $B$ is infinite, 
let $\triangle_B: \N \times \R \to [0,\infty)$ be given by $\triangle_B(j,t) = \triangle_{I_j}(t)$.
Let $1 \leq p \leq \infty$.
If $p < \infty$, say that the barcode $B$ is $p$-finite if it is finite or if $\norm{\triangle_B}_p < \infty$.
If $p = \infty$, say that $B$ is $p$-finite if it is finite or if $\lim_{j \to \infty} \norm{\triangle_{I_j}}_{\infty} \to 0$.
Let $\mathcal{B}_p$ denote the set of such $p$-finite barcodes. 
Then $(\mathcal{B}_p,\Wpt)$ is a metric space.

Our main result is the following.
\begin{theorem*}[Theorem~\ref{thm:embedding-barcode-countable}]
    Let $1 \leq p \leq \infty$.
    The persistence landscape gives a $1$-Lipschitz embedding,
    \begin{equation*}
        \Lambda: (\mathcal{B}_p,\Wpt) \to L^p(\N \times \R).
    \end{equation*}
\end{theorem*}

This map restricts to a $1$-Lipschitz embedding of the space of finite barcodes $\mathcal{B}$.
Similarly, we let $\mathcal{D}_p$ denote the set of $p$-finite persistence diagrams.

\begin{theorem*}[Theorem~\ref{thm:embedding-pd-countable}]
    Let $1 \leq p \leq \infty$.
    The persistence landscape gives a $1$-Lipschitz embedding,
    \begin{equation*}
        \Lambda: (\mathcal{D}_p,\Wpt) \to L^p(\N \times \R).
    \end{equation*}
\end{theorem*}
This map also restricts to a $1$-Lipschitz embedding of the space finite persistence diagrams $\mathcal{D}$.
Note that for $p=2$, we obtain explicit embeddings into a separable Hilbert space, which is of particular interest in topological data analysis, as the Hilbert space structure facilitates the use of statistics and machine learning.

\subsection*{Related work}

Our results are indebted to previous work on Wasserstein distances of barcodes and persistence diagrams 
by Bubenik, Elchesen, and Hartsock
\cite{BE2022,MR4496687,bubenik2026rankbaseddistanceintervalmodules}, whose results depend on earlier work 
by Mileyko et al and Blumberg et al
\cite{Mileyko2011,Blumberg:2014}.
More details on the properties of our constructions may be found in these papers.
For more details on the persistence landscape, see the papers of Bubenik, Edwards and Betthauser
\cite{pl,MR4338670,Betthauser:2022aa}.
The cases $p=1$ and $p=\infty$ of the theorems above are due to Bubenik and Zhao
\cite{pl,bubenik2026rankbaseddistanceintervalmodules}.

The question of whether or not various spaces of barcodes and persistence diagrams embed into Hilbert space has been considered by numerous authors.
Most of these previous works consider the Wasserstein distance using distances obtained from the $q$ norms on $\R^2$.
Turner et al~\cite{tmmh:frechet-means} show that since $(\mathcal{D},W_p)$ has arbitrarily close points with distinct geodesics, it does not embed isometrically into any CAT($k$) space for $k>0$, and hence does not isometrically embed into any Hilbert space.

Carri\`ere and Bauer~\cite{MR3968607} show that there does not exist a bi-Lipschitz embedding of persistence diagrams with at most $n$ points into a finite dimensional Hilbert space.
Carri\`ere, Cuturi and Oudot~\cite{Carriere:2017} give a bi-Lipschitz embedding of bounded persistence diagrams with at most n points into an RKHS.
Bate and Garcia Pulido~\cite{MR4801849} give a bi-Lipschitz embedding of persistence diagrams with at most n points into Hilbert space. 
They also show that the distortion of any such embedding goes to $\infty$ as $n \to \infty$.
Mitra and Virk~\cite{mitra2025geometricembeddingsspacespersistence} give a $1$-Lipschitz embedding of persistence diagrams with at most $n$ points into a finite-dimensional Hilbert space with explicit distortion.

Bell et al~\cite{MR4239889} show that $(\mathcal{D},W_p)$ does not have Yu's property $A$ (this property implies the existence of a coarse embedding into Hilbert space).
Bubenik, Wagner, Mitra and Virk \cite{BubenikWagner:2020,Wagner:2021,MR4246818} show that $(\mathcal{D},W_p)$ does not coarsely embed into Hilbert space for $p>2$.
Pritchard and Weighill~\cite{Pritchard_2024} show that the coarse embeddability into Hilbert space for $(\mathcal{D},W_p)$ is equivalent to that for Borel probability distributions in $\R^2$ with finite $p$th moment.

%Divol and Lacombe~\cite{Divol:2021} give $1$-Lipschitz maps from 

Divol and Lacombe~\cite{Divol:2021} and
Bubenik and Elchesen~\cite{Bubenik:2024b} show that $(\mathcal{D}_1,W_1)$ isometrically embeds into the dual space of compactly supported Lipschitz functions on $\hp$ that vanish on $\diag$.

%\section{Background} \label{sec:background}

\section{Definitions and constructions} \label{sec:constructions}

\subsection{Intervals and ordered pairs}

Let $\Int(\R)$ denote the set of intervals in $\R$.
Let $\Int(\R)_{co}$ denote the set of `closed--open' intervals in $\R$ of the form $[b,d)$, where $-\infty < b < d \leq \infty$ together with the empty interval.
Let $\Int(\R)_{bco}$ denote the set of `bounded closed--open' intervals in $\R$ of the form $[b,d)$, where $-\infty < b < d < \infty$ together with the empty interval.
Let $\hp = \{ (b,d) \in \R^2 \ | \ b \leq d\}$.
Let $\diag = \{ (b,b) \in \R^2\}$.
Let $\ehp = \{ (b,d) \in [-\infty,\infty]^2 \ | \ b \leq d\}$.
Let $\ediag = \{ (b,b) \in [-\infty,\infty]^2\}$.
We have a commutative diagram as follows.
\begin{equation} \label{cd:sets}
    \begin{tikzcd}
        \Int(\R)_{bco} \ar[d,"\isom"] \ar[r,hook] & \Int(\R)_{co} \ar[r,hook] & \Int(\R) \ar[d,two heads]\\
        \hp/\diag \ar[rr,hook] & & \ehp/\ediag
    \end{tikzcd}
\end{equation}
The left vertical arrow sends $[b,d)$ to $(b,d)$ and the empty interval to $\Delta$. 
The right vertical arrow sends a nonempty interval $I$ to $(\inf I, \sup I)$ and sends the empty interval to $\ediag$.
The bottom horizontal arrow sends $\diag$ to $\ediag$ and otherwise sends $(b,d)$ to $(b,d)$.
Using this injective map, we will consider $\hp/\diag$ to be a subset of $\ehp/\ediag$.

\subsection{Distances for intervals and ordered pairs} \label{sec:dq}

For $A \subseteq \R$, the map $d_A:\R \to [0,\infty]$ is given by $d_A(t) = \inf_{a \in A} \abs{t-a}$.
As a special case, $d_\emptyset(t) = \infty$.

\begin{definition} \label{def:triangle}
    For an interval $I$, define the \emph{triangle function} for $I$, $\triangle_I: \R \to [0,\infty]$ by $\triangle_I = d_{I^c}$, where $I^c = \R \setminus I$.
    See Figure~\ref{fig:triangle}.
    Note that $\triangle_\emptyset$ is the constant function with value $0$, as is $\triangle_{\{b\}}$ for each $b \in \R$, and $\triangle_{\R}$ is the constant function with value $\infty$.

    We also define the triangle function for $x=(b,d) \in \ehp$. 
    Let $\triangle_x: \R \to [0,\infty]$ be given by $\triangle_x = \triangle_I$, where $I$ is the interval $(b,d) \subseteq \R$, and $(b,b)$ is the empty interval. 
\end{definition}

\begin{figure}
    \centering
    \begin{tikzpicture}%[scale=1.0]
        \draw (0,0) -- (4,0);
        \draw[very thick,blue] (0,0) -- (1,0) -- (2,1) -- (3,0) -- (4,0);
        \filldraw[black] (1,0) circle (0pt) node[anchor=north]{$b$};
        \filldraw[black] (3,0) circle (0pt) node[anchor=north]{$d$};
    \end{tikzpicture}
    \quad \quad \quad
    \begin{tikzpicture}%[scale=1.0]
        \draw (0,0) -- (4,0);
        \draw[very thick,blue] (0,0) -- (2,0) -- (4,2);
        \filldraw[black] (2,0) circle (0pt) node[anchor=north]{$b$};
    \end{tikzpicture}
    \caption{Graphs of triangle functions $\triangle_I$. Left: $I = [b,d)$. Right: $I = [b,\infty)$. }
    \label{fig:triangle}
\end{figure}
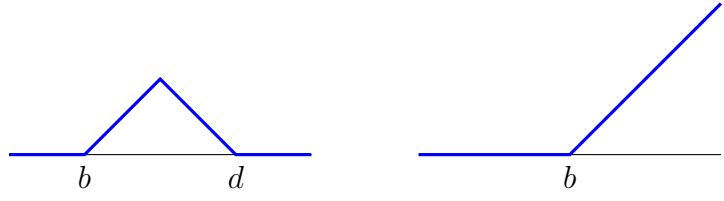

\begin{definition} \label{def:dq}
    Let $1 \leq q \leq \infty$.
    Define $d_q: \Int(\R) \times \Int(\R) \to [0,\infty]$ by 
    \begin{equation*} 
        d_q(I,J) = \norm{\triangle_I - \triangle_J}_q.
    \end{equation*}
    Similarly, define $d_q:\ehp \times \ehp \to [0,\infty]$ by $d_q(x,y) = \norm{\triangle_x - \triangle_y}_q$.
\end{definition}

It is easy to check that each $d_q$ is reflexive, symmetric, and satisfies the triangle inequality for the sets in $\eqref{cd:sets}$. 
Thus we have the following.

\begin{proposition}
    Let $1 \leq q \leq \infty$.
    Then
    \begin{enumerate}
        \item $d_q$ is an extended pseudometric on $\Int(\R)$.
        \item $d_q$ is an extended metric on $\Int(\R)_{co}$.
        \item $d_q$ is a metric on $\Int(\R)_{bco}$.
        \item $d_q$ is an extended metric on $\ehp/\ediag$.
        \item $d_q$ is a metric on $\hp/\diag$.
    \end{enumerate}
\end{proposition}

%A morphism of extended pseudometric spaces is a $1$-Lipschitz map.
Assigning $d_q$ to each of the sets in \eqref{cd:sets}, we obtain a commutative diagram of extended pseudometric spaces and $1$-Lipschitz maps.

\subsection{Finite barcodes and persistence diagrams}

For each of the sets in \eqref{cd:sets}, we choose a distinguished element. 
For the top three sets, it is the empty interval.
For the bottom sets, it is the diagonal, $\diag$ and $\ediag$ respectively.
With these points, \eqref{cd:sets} becomes a commutative diagram of pointed sets. 
Furthermore,
giving these pointed sets the distances $d_q$, \eqref{cd:sets} becomes a commutative diagram of pointed extended pseudometric spaces and $1$-Lipschitz maps.
We write $(X,d,x_0)$ for a pointed extended pseudometric space and $f:(X,d_X,x_0) \to (Y,d_Y,y_0)$ for a $1$-Lipschitz map of pointed extended pseudometric spaces.

A formal sum on a set $S$ is also called a finite multiset on $S$.
We will denote such a sum as $\sum_{i=1}^n s_i$, where $s_i \in S$ and $n \geq 0$.
If $n=0$ then we write this sum as $0$.
We denote the set of all such formal sums as $D(S)$.%
\footnote{Given a set $S$, $D(S)$ denotes the free commutative monoid on the set $S$.}
Given a pointed set $(S,s_0)$, let $D(S,s_0) = D(S \setminus \{s_0\})$.%
\footnote{More formally, $D(S,s_0) = D(S)/D(\{s_0\})$, which is isomorphic to $D(S \setminus \{s_0\})$.}
Given a map $f:(X,x_0) \to (Y,y_0)$ there is an induced map $Df: D(X,x_0) \to D(Y,y_0)$ given by sending $\sum_{i=1}^n x_i$ to $\sum_{i=1}^n f(x_i)$.%
\footnote{If we equip $D(X,x_0)$ with the basepoint $0$, then $D$ is an endofunctor on pointed sets.}

A \emph{finite barcode} is an element of $D(\widehat{\Int}(\R))$, where $\widehat{\Int}(\R) = \Int(\R) \setminus \{\emptyset\}$.
That is, it is a formal sum of nonempty intervals.
A \emph{finite persistence diagram} is an element of 
$D(\ehp/\ediag),\ediag) = D(\oehp)$,
where $\oehp = \{ (b,d) \in [-\infty,\infty]^2 \ | \ b < d\}$.

% A \emph{barcode} is a formal sum of intervals (also called a finite multiset).
% Let $D(\Int(\R))$ denote the set of barcodes.
% A \emph{persistence diagram} is a formal sum of elements of $\ehp/\ediag$ or a formal sum of elements of $\hp/\diag$.
% Let $D(\ehp/\ediag)$ and $D(\hp/\diag)$ denote the corresponding sets of persistence diagrams.

\subsection{Wasserstein distance}

Let $1 \leq p \leq \infty$.

Given a pointed extended pseudometric space $(X,d,x_0)$,
%the set of formal sums on the complement of the basepoint, 
$D(X,x_0)$ has an extended pseudometric $W_p(d)$, called the $p$-Wasserstein distance induced by $d$, given as follows.
\begin{equation*}
    W_p(d)\left(\sum_{i=1}^m x_i, \sum_{i=1}^n y_i\right) = \min_{\sigma \in \Sigma_{m+n}} \norm{ (d(x_i,y_{\sigma(i)}))_{i=1}^{m+n} }_p,
\end{equation*}
where $x_{m+1} = \cdots x_{m+n} = x_0 = y_{n+1} = \cdots y_{n+m}$
and $\Sigma_{m+n}$ denotes the symmetric group on $\{1,\ldots,m+n\}$.
We write $D(X,d,x_0) = (D(X \setminus \{x_0\}), W_p(d), 0)$.
Given a map $f:(X,d_X,x_0) \to (Y,d_Y,y_0)$, which is $1$-Lipschitz by definition,
the map $Df: D(X,x_0) \to D(Y,y_0)$ is $1$-Lipschitz.
That is, we have $Df: D(X,d_X,x_0) \to D(Y,d_Y,y_0)$.%
\footnote{$D$ is an endofunctor on pointed extended pseudometric spaces.}

% Given a distance $d$ on $\Int(\R)$, there is an induced $p$-Wasserstein distance $W_p(d)$ on $D(\Int(\R)$.
% Similarly, given a distance $d$ on $\ehp/\ediag$ or $\hp/\diag$, there is a $p$-Wasserstein distance, $W_p(d)$ on $D(\ehp/\ediag)$ or $D(\hp/\diag)$.
% These distances are reflexive, symmetric, and satisfy the triangle inequality.

\begin{definition} \label{def:Wpt}
%    Let $1 \leq p \leq \infty$.
    Finite barcodes and finite persistence diagrams have the $p$-Wasserstein distance $\Wpt = W_p(d_p)$.
    Thus, for 
    $\alpha = \sum_{i=1}^m (b_i,d_i), \alpha' = \sum_{i=1}^n (b'_i,d'_i) \in D(\oehp)$,
    $\Wpt(\alpha,\alpha') = \Wpt(B,B')$, where 
    $B = \sum_{i=1}^m (b_i,d_i), B' = \sum_{i=1}^n (b'_i,d'_i) \in D(\widehat{\Int}(\R))$.
\end{definition}

\begin{proposition}
%    Let $1 \leq p \leq \infty$.
    \begin{enumerate}
        \item $\Wpt$ is an extended pseudometric on $D(\widehat{\Int}(\R))$.
        \item $\Wpt$ is an extended metric on $D(\widehat{\Int}(\R)_{co})$.
        \item $\Wpt$ is a metric on $D(\widehat{\Int}(\R)_{bco})$.
        \item $\Wpt$ is an extended metric on $D(\oehp)$.
        \item $\Wpt$ is a metric on $D(\ohp)$.
    \end{enumerate}
\end{proposition}

Applying $D$ to \eqref{cd:sets}, viewed as a commutative diagram of pointed extended pseudometric spaces and $1$-Lipschitz maps, we have the following commutative diagram of extended pseudometric spaces and $1$-Lipschitz maps.
\begin{equation*} \label{cd:Wpt}
    \begin{tikzcd}
        (D(\widehat{\Int}(\R)_{bco}),\Wpt) \ar[d,"\isom"] \ar[r,hook] & (D(\widehat{\Int}(\R)_{co}),\Wpt) \ar[r,hook] & (D(\widehat{\Int}(\R)),\Wpt) \ar[d,two heads]\\
        (D(\ohp),\Wpt) \ar[rr,hook] & & (D(\oehp),\Wpt)
    \end{tikzcd}
\end{equation*}

\subsection{Persistence landscapes}

Given a sequence $a = (a_1,\ldots,a_n)$ of nonnegative numbers and $1 \leq k \leq n$,
let $\kmax a$ be the $k$th largest element of $a$.
For $k > n$, let $\kmax a$ equal $0$.

\begin{definition} \label{def:pl}
    Define $\Lambda: D(\widehat{\Int}(\R)) \to [0,\infty]^{\N \times \R}$ as follows.
    For the finite barcode $B = \sum_{j=1}^n I_j \in D(\widehat{\Int}(\R))$,
    \begin{equation*}
        \Lambda_B(k,t) = \kmax (\triangle_{I_j}(t))_{j=1}^n.
    \end{equation*}
    Call $\Lambda_B$ the \emph{persistence landscape} of $B$.
    Similarly, define $\Lambda: D(\oehp) \to [0,\infty]^{\N \times \R}$ %and $\Lambda: D(\hp/\diag) \to [0,\infty]^{\N \times \R}$ 
    as follows.
    For the finite persistence diagram $\alpha = \sum_{j=1}^n x_j$,
    the persistence landscape of $\alpha$ is given by
    \begin{equation*}
        \Lambda_\alpha(k,t) = \kmax (\triangle_{x_j}(t))_{j=1}^n.
    \end{equation*}
    Thus, for $\alpha = \sum_{j=1}^n (b_j,d_j) \in D(\oehp)$, 
    $\Lambda_\alpha = \Lambda_B$, where $B = \sum_{j=1}^n (b_j,d_j) \in D(\widehat{\Int}(\R)$.
\end{definition}

\section{Results for finite barcodes and persistence diagrams}

Let $1 \leq p \leq \infty$. 
The following is the main result for this section.

\begin{theorem} \label{thm:main}
    Let $B,B' \in D(\widehat{\Int}(\R))$.
    \begin{equation*}
        \norm{\Lambda_B-\Lambda_{B'}}_p \leq \Wpt(B,B').
    \end{equation*}
\end{theorem}

First we prove a special case for which we have equality.

\begin{proposition}
    Let $B=I$ and $B' = 0$. Then $\norm{\Lambda_B - \Lambda_{B'}}_p = \norm{\triangle_I}_p =  \Wpt(B,B')$.
\end{proposition}

\begin{proof}
    $\Lambda_B(k,t) - \Lambda_{B'}(k,t) = \triangle_I(t)$ if $k=1$ and $0$ otherwise.
    $\Wpt(B,B') = d_p(I,\emptyset) = \norm{\triangle_I}_p$.
\end{proof}

For the proof of Theorem~\ref{thm:main}, we need a lemma.
Given a sequence $a = (a_i)_{i=1}^n$, for $1 \leq k \leq n$, let $a_{(k)} = \kmin a$, where $\kmin a$ is the $k$th smallest element of the sequence $a$.
That is, we have reordered to the sequence $a$ to obtain the sequence $(a_{(k)})_{k=1}^n$ such that $a_{(1)} \leq \cdots \leq a_{(n)}$.

\begin{lemma} \label{lem:ot}
    Given sequences $(a_i)_{i=1}^n$ and $(b_i)_{i=1}^n$,
    \begin{equation*}
        \norm{(a_{(i)} - b_{(i)})_{i=1}^n}_p \leq \norm{(a_i-b_i)_{i=1}^n}_p.
    \end{equation*}
\end{lemma}

\begin{proof}
    First, assume that $p < \infty$.
    Since the function $\abs{t}^p$ is convex, 
    optimal transportation on the real line is given by a monotone rearrangement.
    That is,
    \[
    \sum_{i=1}^n \abs{a_{(i)} - b_{(i)}}^p \leq \sum_{i=1}^n \abs{a_i-b_i}^p.
    \]
    Taking the $p$th root, we obtain the desired result.
    Taking the limit as $p \to \infty$, we obtain the remaining case.
\end{proof}

\begin{proof}[Proof of Theorem~\ref{thm:main}]
    Let $B = \sum_{i=1}^m I_i$ and $B' = \sum_{i=1}^n I'_i$.
    Choose $\sigma \in \Sigma_{m+n}$ such that 
    \begin{equation*}
        \Wpt(B,B') = \norm{ (d_p(I_i,I'_{\sigma(i)}))_{i=1}^{m+n}}_p,
    \end{equation*}
    where $I_{m+1} = \cdots I_{m+n} = \emptyset = I'_{n+1} = \cdots = I'_{n+m}$.
    Using Definition~\ref{def:dq}, Tonelli's theorem, Lemma~\ref{lem:ot}, and Tonelli's theorem again, we have the following.
    \begin{align*}
        \Wpt(B,B') 
        &= \norm{ ( \norm{ ( \triangle_{I_i} - \triangle_{I'_{\sigma(i)}} }_p )_{i=1}^{m+n} }_p \\
        &= \norm{ \norm{ ( \triangle_{I_i}(t) - \triangle_{I'_{\sigma(i)}}(t) )_{i=1}^{m+n} }_p }_p \\
        &\geq \norm{ \norm{ ( \Lambda_B(k,t) - \Lambda_{B'}(k,t) )_{k=1}^{m+n} }_p }_p \\
        &= \norm{ \Lambda_B - \Lambda_{B'} }_p. \qedhere
    \end{align*}
\end{proof}
 
From our main result, we obtain the following three results.

\begin{theorem}
    Let $\alpha,\alpha' \in D(\oehp)$.
    \begin{equation*}
        \norm{\Lambda_\alpha-\Lambda_{\alpha'}}_p \leq \Wpt(\alpha,\alpha').
    \end{equation*}
\end{theorem}

\begin{proof}
    Let $B,B'$ be the finite barcodes corresponding to the finite persistence diagrams $\alpha,\alpha'$, as in Definition~\ref{def:Wpt}, which gives $\Wpt(\alpha,\alpha') = \Wpt(B,B')$.
    Also, from Definition~\ref{def:pl}, $\Lambda_\alpha = \Lambda_B$ and $\Lambda_{\alpha'} = \Lambda_{B'}$.
    The result now follows from Theorem~\ref{thm:main}.
\end{proof}

Give $\N \times \R$ the product measure $\mu \times \nu$ of the counting measure $\mu$ and the Lebesgue measure $\nu$. It is uniquely defined since $\mu$ and $\nu$ are $\sigma$-finite.
The Lebesgue space $L^p(\N \times \R)$ is a Banach space and it is separable if $p < \infty$.

\begin{theorem} \label{thm:embedding-barcode}
    The persistence landscape gives a $1$-Lipschitz embedding,
    \begin{equation*}
        \Lambda: (D(\widehat{\Int}(\R)_{bco}),\Wpt) \to L^p(\N \times \R).
    \end{equation*}
\end{theorem}

\begin{proof}
    Let $B,B' \in D(\widehat{\Int}(\R)_{bco})$.
    By Theorem~\ref{thm:main},
    $\norm{\Lambda_B-\Lambda_{B'}}_p \leq \Wpt(B,B')$.
    Injectivity of $\Lambda$ is observed in \cite{pl}.
\end{proof}

\begin{theorem} \label{thm:embedding-pd}
    The persistence landscape gives a $1$-Lipschitz embedding,
    \begin{equation*}
        \Lambda: (D(\ohp),\Wpt) \to L^p(\N \times \R).
    \end{equation*}
\end{theorem}

\begin{proof}
    Let $\alpha,\alpha' \in D(\ohp)$.
    Let $B,B' \in D(\widehat{\Int}(\R)_{bco})$ be the corresponding finite barcodes.
    Since $\Lambda_\alpha = \Lambda_B$, $\Lambda_{\alpha'} = \Lambda_{B'}$, and $\Wpt(\alpha,\alpha') = \Wpt(B,B')$, the result now follows from the previous theorem.
\end{proof}

\section{Results for countable barcodes and persistence diagrams}

Let $1 \leq p \leq \infty$.
In this section, we will extend the $1$-Lipschitz embeddings in Theorems \ref{thm:embedding-barcode} and \ref{thm:embedding-pd} to the completions of the metric spaces $(D(\widehat{\Int}(\R)),\Wpt)$ and $(D(\ohp),\Wpt)$.

%\subsection{Countable barcodes and persistence diagrams} % and their persistence landscapes}

%Let $(S,d,s_0)$ be a pointed metric space.
A countable formal sum on a set $S$ is given by $\sum_{i \in I} s_i$, where $s_i \in S$ and $I$ is countable.
Let $\overline{D}(S)$ denote the set of countable formal sums on $S$.
A \emph{countable barcode} is an element of 
$\overline{D}(\widehat{\Int}(\R))$ and
a \emph{countable persistence diagram} is an element of 
$\overline{D}(\oehp)$

We extend the definition of persistence landscapes (Definition~\ref{def:pl}) to countable barcodes and persistence diagrams.

\begin{definition} \label{def:pl-countable}
    We extend the maps in Definition~\ref{def:pl} to define
    $\Lambda: \overline{D}(\widehat{\Int}(\R)) \to [0,\infty]^{\N \times \R}$ and
    $\Lambda: \overline{D}(\oehp) \to [0,\infty]^{\N \times \R}$ as follows.
    For $B = \sum_{j=1}^\infty I_j \in \overline{D}(\widehat{\Int}(\R))$, let%
    \footnote{We need to take some care to define $\kmax$ for infinite sequences, since the sequence may not have a $k$th largest element. Consider a sequence $(a_i)_{i=1}^\infty$, with $a_i \geq 0$.
    Let $L = \lim \sup a \in [0,\infty]$. 
    Let $b$ be the subsequence of $a$ consisting of elements larger that $L$. 
    If $b$ has $N$ terms then define $\kmax a = \kmax b$ if $k \leq N$ and otherwise $\kmax a = L$.
    If $b$ is infinite, then for each $k$, $a$ has a $k$th largest element. 
    For $p$-finite sequences, their limit is $0$ and they have a $k$th largest element for all $k$.}
    \begin{equation*}
        \Lambda_B(k,t) = \kmax (\triangle_{I_j}(t))_{j=1}^\infty
    \end{equation*}
    For $\alpha = \sum_{j=1}^\infty x_j \in \overline{D}(\oehp)$, let
    \begin{equation*}
        \Lambda_\alpha(k,t) = \kmax (\triangle_{x_j}(t))_{j=1}^\infty
    \end{equation*}
\end{definition}

Let $(S,d,s_0)$ be a pointed metric space.
If $p < \infty$ then the countable formal sum on $S \setminus \{s_0\}$,
$\sum_{i \in I} s_i$, is said to be \emph{$p$-finite} for $d$ if 
$\norm{(d(s_i,s_0))_{i \in I}}_p < \infty$.
If $p=\infty$ then it is said to be $p$-finite for $d$ if 
%for every $\eps > 0$ the number of $s_i$ for which $d(s_i,s_0) > \eps$ is finite.
$\lim_{i \to \infty} d(s_i,s_0) = 0$.
Let $\overline{D}_p^d(S,s_0) = \overline{D}_p^d(S \setminus \{s_0\})$ be the set of  countable formal sums on $S \setminus \{s_0\}$ that are $p$-finite for $d$. 
We extend $W_p(d)$ from $D(S,s_0)$ to $\overline{D}_p^d(S,s_0)$ as follows.
For $\sum_{i=1}^\infty s_i, \sum_{i=1}^\infty s'_i \in \overline{D}_p^d(S,s_0)$,
let 
\[
    W_p(d) \left( \sum_{i=1}^\infty s_i, \sum_{i=1}^\infty s'_i \right) = \lim_{n \to \infty} W_p(d) \left( \sum_{i=1}^n s_i, \sum_{i=1}^n s'_i \right).
\]

\begin{lemma}[{\cite[Theorem 6.20, Lemma 6.17]{MR4496687}}]
    $W_p(d)$ is a complete metric on $\overline{D}_p^d(S,s_0)$ and $(\overline{D}_p^d(S,s_0),W_p(d))$ is a completion of $(D(S,s_0),W_p(d))$.
\end{lemma}

\begin{lemma}[{\cite[Theorem 6.2]{MR4768640}}]
    If $(S,d,s_0)$ is separable then so is $(\overline{D}_p^d(S,s_0),W_p(d))$.
\end{lemma}

\begin{definition}
    Write $\overline{D}_p^\triangle$ in place of $\overline{D}_p^{d_p}$.
    We have a complete separable metric space of $p$-finite countable barcodes $(\overline{D}_p^\triangle(\widehat{\Int}(\R)_{bco}),\Wpt)$
    and a complete separable metric space of $p$-finite countable persistence diagrams $(\overline{D}_p^\triangle(\ohp),\Wpt)$.
    The map sending the interval $[b,d)$ to the point $(b,d)$ induces an isometry between these metric spaces.
\end{definition}

%\subsection{Completions of metric spaces of barcodes and persistence diagrams}

Let $X$ and $Y$ be metric spaces and let $\overline{X}$ be a completion of $X$.
A $1$-Lipschitz map $f:X \to Y$ between metric spaces $X$ and $Y$ extends to a $1$-Lipschitz map $\overline{f}: \overline{X} \to Y$ given by defining $\overline{f}(x) = \lim_{n\to \infty} f(x_n)$, where $(x_n)$ is any sequence in $X$ that converges in $\overline{X}$ to $x$.
From this construction, we obtain the following two results.

\begin{theorem} \label{thm:embedding-barcode-countable}
    The persistence landscape gives a $1$-Lipschitz embedding,
    \begin{equation*}
        \Lambda: (\overline{D}_p^\triangle(\widehat{\Int}(\R)_{bco}),\Wpt) \to L^p(\N \times \R).
    \end{equation*}
\end{theorem}

\begin{proof}
    The persistence landscape for a $p$-finite countable barcode (Definition~\ref{def:pl-countable}) is the $1$-Lipschitz extension of the persistence landscape for finite barcodes (Definition~\ref{def:pl}).
\end{proof}

\begin{theorem} \label{thm:embedding-pd-countable}
    The persistence landscape gives a $1$-Lipschitz embedding,
    \begin{equation*}
        \Lambda: (\overline{D}_p^\triangle(\ohp),\Wpt) \to L^p(\N \times \R).
    \end{equation*}
\end{theorem}

\begin{proof}
    The persistence landscape for a $p$-finite countable persistence diagrams (Definition~\ref{def:pl-countable}) is the $1$-Lipschitz extension of the persistence landscape for finite persistence diagrams (Definition~\ref{def:pl}).
\end{proof}

\section{Additional results}

\subsection{The p-finite condition}

In this section we compare the conditions that an countable barcode or persistence diagram is $p$-finite for $d_p$ and $p$-finite for the metric induced by the $p$ norm.

Let $1 \leq p \leq \infty$.
Let $\ell_p$ denote the quotient metric on $\hp/\diag$ obtained from the metric on $\hp$ given by the $p$ norm.
Similarly, also let $\ell_p$ denote the corresponding metric on $\Int(\R)_{bco}$.

Consider the barcode $B = \sum_{j=1}^\infty I_j$.
Let $h = (\frac{1}{2} \nu(I_j))_{j=1}^\infty$, where $\nu(I)$ is the Lebesgue measure (i.e. length) of the interval $I$.

To start, we consider the case $p = \infty$.
For an interval $I$, $d_\infty(I,\emptyset) = \norm{\triangle_I}_\infty = \frac{1}{2} \nu(I)$ and $\ell_\infty(I,\emptyset) = \frac{1}{2} \nu(I)$.
Therefore we have the following.

\begin{proposition}
    The following three conditions are equivalent.
    \begin{enumerate}
        \item $B$ is $\infty$-finite for $d_\infty$
        \item $B$ is $\infty$-finite for $\ell_\infty$.
        \item $\lim h = 0$.
    \end{enumerate}
\end{proposition}

% \begin{proof}
%     $B$ is $\infty$-finite for both $d_\infty$ and $\ell_\infty$ if and only if $\lim \nu(I) = 0$.
% \end{proof}

We also have the corresponding result for countable persistence diagrams.
Now assume that $p < \infty$.

\begin{proposition}
    \begin{enumerate}
        \item $B$ is $p$-finite with respect to $d_p$ if and only if $\norm{h}_{p+1} < \infty$.
        \item $B$ is $p$-finite with respect to $\ell_p$ if and only if $\norm{h}_p < \infty$.
    \end{enumerate}
\end{proposition}

\begin{proof}
    For an interval $I$, %let $h = \frac{1}{2}\nu(I)$.
    %Then 
    $\norm{\triangle_I}_p = (2 \int_0^{\nu(I)/2} t^p \ dt)^{\frac{1}{p}} = (\frac{2}{p+1})^{\frac{1}{p}} (\frac{1}{2}\nu(I))^{\frac{p+1}{p}}$.
    Recall that $h_j = \frac{1}{2} \nu(I_j)$.
    Thus $B$ is $p$-finite for $d_p$ 
    iff $\norm{(h_j^{\frac{p+1}{p}})}_p < \infty$
    iff $\norm{(h_j^{\frac{p+1}{p}})}_p^p < \infty$
    iff $\sum_{j=1}^\infty h_j^{p+1} < \infty$
    iff $\norm{h}_{p+1} < \infty$.

    For an interval $I$, 
    $\ell_p(I,\emptyset) = \norm{(\frac{1}{2}\nu(I),\frac{1}{2}\nu(I))}_p = 2^{\frac{1}{p}} \frac{1}{2} \nu(I)$.
    Thus $B$ is $p$-finite for $\ell_p$ 
    iff $\norm{h}_p < \infty$.
\end{proof}

We also have the corresponding result for countable persistence diagrams.
As a corollary to the above two Propositions we have the following.

\begin{theorem}
    \begin{enumerate}
    \item $\overline{D}_\infty^\triangle(\ohp) = \overline{D}_\infty^{\ell_\infty}(\ohp)$ and
    $\overline{D}_\infty^\triangle(\widehat{\Int}(\R)_{bco}) = \overline{D}_\infty^{\ell_\infty}(\widehat{\Int}(\R)_{bco})$.
    \item 
    For $1 \leq p < \infty$,
    $\overline{D}_p^\triangle(\ohp) = \overline{D}_{p+1}^{\ell_{p+1}}(\ohp)$ and
    $\overline{D}_p^\triangle(\widehat{\Int}(\R)_{bco}) = \overline{D}_{p+1}^{\ell_{p+1}}(\widehat{\Int}(\R)_{bco})$.
    \item
    For $1 \leq p < \infty$,
    $\overline{D}_p^\triangle(\ohp) \supsetneq \overline{D}_p^{\ell_p}(\ohp)$  and
    $\overline{D}_p^\triangle(\widehat{\Int}(\R)_{bco}) \supsetneq \overline{D}_p^{\ell_p}(\widehat{\Int}(\R))_{bco}$.
    \end{enumerate}
\end{theorem}

\subsection{Relating Wasserstein distances}

In this section we show that for the cases $p=1$ and $p=\infty$, our Wasserstein distances $\Wpt$ agree with previous definitions.

\begin{proposition} \label{prop:bottleneck}
    If $p=\infty$ then $\Wpt$ coincides with the bottleneck distance~\cite{bottleneck}.
\end{proposition}

\begin{proof}
    Let $\ell_{\infty}$ denote the quotient metric on $\hp/\diag$ obtained from the metric on $\hp$ given by the $\infty$ norm.
    Let $(\ell_\infty)_\infty$ denote the $\infty$-strengthening of the $\ell_\infty$ metric~\cite{BE2022}.
    That is, for $x,y \in \hp$, 
    $(\ell_\infty)_\infty(x,y) = \min( \norm{x-y}_\infty, \max( \norm{x-\pi(x)}_\infty, \norm{y-\pi(y)}_\infty) )$,
    where $\pi(b,d) = (\frac{b+d}{2},\frac{b+d}{2})$.
    Then $d_\infty = (\ell_\infty)_\infty$.
    Therefore, $W_\infty^\triangle = W_\infty(d_\infty) = W_\infty((\ell_\infty)_\infty) = W_\infty(\ell_\infty)$,
    where the last equality is due to \cite[Corollary 5.3]{BE2022},
    and the last distance is the bottleneck distance.
\end{proof}

\begin{proposition} \label{prop:rank}
    If $p=1$ then $\Wpt$ coincides with the rank-based Wasserstein distance~\cite{bubenik2026rankbaseddistanceintervalmodules}.
\end{proposition}

\begin{proof}
    By \cite[Lemma 3.2]{bubenik2026rankbaseddistanceintervalmodules},
    $d_1$ %, which they denote $d_\Lambda$ 
    equals $d_{\rank}$.
    Therefore $W_1^\triangle = W_1(d_1) = W_1(d_{\rank}) = W_1^{\rank}$.
\end{proof}

\section*{Acknowledgments}

This research was partially supported 
by the National Science Foundation (NSF) grant DMS-2324353.

\printbibliography

\end{document}